\documentclass[a4paper,12pt,leqno]{amsart}
\usepackage[english]{babel}
\usepackage[T1]{fontenc}
\usepackage[left=1in,right=1in,top=1.2in,bottom=1.2in]{geometry}
\usepackage{times}
\usepackage{microtype}

\usepackage{commath,amssymb,amscd,mathrsfs,mathtools,dsfont,bbm,multirow,array,caption,comment,pifont}
\usepackage[all]{xy}
\usepackage{enumerate}
\usepackage{longtable}

\usepackage{cite}			

\usepackage[usenames,dvipsnames,svgnames,table]{xcolor}
\definecolor{red}{RGB}{250,25,25}
\definecolor{blue}{RGB}{50,50,200}
\usepackage{tikz-cd}
\usepackage[pagebackref,linktocpage]{hyperref}

\hypersetup{
pdftitle={DDC-AV},		
pdfauthor={Fei Hu},		
colorlinks=true,			
linkcolor=red,			
citecolor=MidnightBlue,	
filecolor=magenta,		
urlcolor=MidnightBlue	
}

\usepackage{cleveref} 

\newtheorem{theorem}{Theorem}[section]
\crefname{theorem}{Theorem}{Theorems}
\newtheorem{lemma}[theorem]{Lemma}
\crefname{lemma}{Lemma}{Lemmas}
\newtheorem{proposition}[theorem]{Proposition}
\crefname{proposition}{Proposition}{Propositions}

\crefname{prop}{Proposition}{Propositions}
\newtheorem{corollary}[theorem]{Corollary}
\crefname{corollary}{Corollary}{Corollaries}

\crefname{cor}{Corollary}{Corollaries}
\newtheorem{conjecture}[theorem]{Conjecture}
\crefname{conjecture}{Conjecture}{Conjectures}

\crefname{conj}{Conjecture}{Conjectures}
\newtheorem*{conj*}{Conjecture}
\crefname{conj}{Conjecture}{Conjectures}

\theoremstyle{definition}
\newtheorem{definition}[theorem]{Definition}
\crefname{definition}{Definition}{Definitions}

\crefname{defn}{Definition}{Definitions}
\newtheorem{example}[theorem]{Example}
\crefname{example}{Example}{Examples}

\crefname{notation}{Notation}{Notation}
\newtheorem*{notation*}{Notation}
\crefname{notation}{Notation}{Notation}

\crefname{problem}{Problem}{Problems}

\crefname{question}{Question}{Questions}

\crefname{condition}{Condition}{Conditions}

\crefname{assumption}{Assumption}{Assumptions}

\theoremstyle{remark}

\crefname{rmk}{Remark}{Remarks}
\newtheorem*{rmk*}{Remark}
\crefname{rmk}{Remark}{Remarks}
\newtheorem{remark}[theorem]{Remark}
\crefname{remark}{Remark}{Remarks}

\crefname{fact}{Fact}{Facts}

\crefname{claim}{Claim}{Claims}
\newtheorem*{claim*}{Claim}
\crefname{claim}{Claim}{Claims}

\crefname{step}{Step}{Steps}
\newtheorem{case}{Case}
\crefname{case}{Case}{Cases}

\numberwithin{equation}{section}

\newcommand{\what}[1]{\widehat{#1}}

\newcommand{\ol}[1]{\overline{#1}}

\newcommand{\ratmap}{\dashrightarrow}

\newcommand{\longinjmap}{\lhook\joinrel\longrightarrow}

\newcommand{\doi}[1]{\href{https://doi.org/#1}{{\tt doi:#1}}}

\newcommand{\bA}{\mathbf{A}}

\newcommand{\bC}{\mathbf{C}}

\newcommand{\bF}{\mathbf{F}}

\newcommand{\bH}{\mathbf{H}}
\newcommand{\bI}{\mathbf{I}}

\newcommand{\bM}{\mathbf{M}}
\newcommand{\bN}{\mathbf{N}}

\newcommand{\bQ}{\mathbf{Q}}
\newcommand{\bR}{\mathbf{R}}
\newcommand{\bT}{\mathbf{T}}
\newcommand{\bU}{\mathbf{U}}

\newcommand{\bZ}{\mathbf{Z}}

\newcommand{\bbi}{\mathbbm{i}}
\newcommand{\bbj}{\mathbbm{j}}
\newcommand{\bbk}{\mathbbm{k}}
\newcommand{\bk}{\mathbf{k}}
\newcommand{\bu}{\boldsymbol{u}}
\newcommand{\bv}{\boldsymbol{v}}

\newcommand{\sF}{\mathscr{F}}

\newcommand{\sH}{\mathscr{H}}

\newcommand{\sL}{\mathscr{L}}
\newcommand{\sM}{\mathscr{M}}
\newcommand{\sO}{\mathscr{O}}

\newcommand{\sT}{\mathsf{T}}

\newcommand{\reduced}{{\rm red}}

\newcommand{\alg}{\operatorname{alg}}
\newcommand{\alb}{\operatorname{alb}}
\newcommand{\Alb}{\operatorname{Alb}}

\newcommand{\Albert}{\mathsf{A}}

\newcommand{\ch}{\operatorname{char}}

\newcommand{\dR}{\operatorname{dR}}

\newcommand{\End}{\operatorname{End}}
\newcommand{\et}{{\textrm{\'et}}}

\newcommand{\Frob}{\mathsf{F}}

\newcommand{\Hom}{\operatorname{Hom}}

\newcommand{\isom}{\simeq}

\newcommand{\Mat}{\operatorname{M}}

\newcommand{\Nef}{\mathsf{Nef}}
\newcommand{\N}{\mathsf{N}}

\newcommand{\Nm}{\operatorname{N}}
\newcommand{\NS}{\mathsf{NS}}

\newcommand{\Pic}{\operatorname{Pic}}

\begin{document}

\title[Eigenvalues and dynamical degrees on abelian varieties]{Eigenvalues and dynamical degrees of self-maps on abelian varieties}

\author{Fei Hu}
\address{Pacific Institute for the Mathematical Sciences, Vancouver, BC V6T1Z4, Canada}
\curraddr{Department of Mathematics, University of Oslo, P.O. Box 1053, Blindern, 0316 Oslo, Norway}


\email{\href{mailto:hf@u.nus.edu}{\tt hf@u.nus.edu}}

\begin{abstract}
Let $X$ be a smooth projective variety over an algebraically closed field, and $f\colon X\to X$ a surjective self-morphism of $X$.
The $i$-th cohomological dynamical degree $\chi_i(f)$ is defined as the spectral radius of the pullback $f^{*}$ on the \'etale cohomology group $H^i_{\textrm{\'et}}(X, \mathbf{Q}_\ell)$ and the $k$-th numerical dynamical degree $\lambda_k(f)$ as the spectral radius of the pullback $f^{*}$ on the vector space $\mathsf{N}^k(X)_{\mathbf{R}}$ of real algebraic cycles of codimension $k$ on $X$ modulo numerical equivalence.
Truong conjectured that $\chi_{2k}(f) = \lambda_k(f)$ for all $0 \le k \le \dim X$ as a generalization of Weil's Riemann hypothesis.
We prove this conjecture in the case of abelian varieties.
In the course of the proof, we also obtain a new parity result on the eigenvalues of self-maps of abelian varieties in prime characteristic, which is of independent interest.
\end{abstract}




\thanks{The author was partially supported by a UBC-PIMS Postdoctoral Fellowship and Young Research Talents grant \#300814 from the Research Council of Norway.}


\maketitle



\section{Introduction}
\label{section:intro}


\subsection{Parity of eigenvalues}

Given a self-morphism $f \colon X \to X$ of a smooth complex projective variety $X$, the characteristic polynomial $P_i(f, t)$ of the pullback $f^*$ on $H^i(X, \bC)$ is a monic polynomial of degree $b_i(X)$ with integer coefficients, where $b_i(X)$ denotes the $i$-th Betti number of $X$.
As a direct consequence of the Hodge decomposition, for odd $i$, the $b_i(X)$ complex roots of $P_i(f, t)$ fall into $b_i(X)/2$ pairs with each one consisting of two conjugate complex numbers; in particular, all real roots of $P_i(f, t)$ are of even multiplicity.
Over a base field of positive characteristic, because of the absence of the Hodge decomposition,
it seems to be unknown whether an analog of the above parity with complex cohomology replaced by $\ell$-adic \'etale cohomology holds in general.

We show that it holds for abelian varieties in arbitrary characteristic.
Henceforth, we fix an algebraically closed field $\bk$ of arbitrary characteristic.

\begin{theorem}
\label{thm:B}
Let $\alpha \colon X \to X$ be an endomorphism of an abelian variety $X$ of dimension $g$ over $\bk$.
Let $P_\alpha(t)$ be the characteristic polynomial of the pullback $\alpha^*$ on the first \'etale cohomology group $H^1_{\emph\et}(X, \bQ_\ell)$ of $X$, which is a monic polynomial of degree $2g$ with integer coefficients.
Then there exists a monic polynomial $P_\alpha^{\Albert}(t) \in \bC[t]$ of degree $g$ such that $P_\alpha(t) = P_\alpha^{\Albert}(t)  \overline{P_\alpha^{\Albert}(t)}$.

In particular, the $2g$ complex roots of $P_\alpha(t)$ fall into $g$ pairs with each one consisting of two conjugate complex numbers; all real roots of $P_\alpha(t)$ are of even multiplicity.
\end{theorem}

Here, the characteristic polynomial $P_\alpha(t)$ of $\alpha$ is equivalently defined as the characteristic polynomial of the induced endomorphism $T_\ell(\alpha)$ on the Tate module $T_\ell(X)$ of $X$, independent of $\ell\neq \ch(\bk)$; see \S\ref{section:prelim} for details on abelian varieties.

\begin{remark}
\label{rmk:B1}
With notation as in \cref{thm:B}, we notice that the pullback $\alpha^*$ on the $i$-th \'etale cohomology group $H^i_{\et}(X, \bQ_\ell)$ is completely determined by its pullback $\alpha^*$ on $H^1_{\et}(X, \bQ_\ell)$ (see, e.g., \cref{thm:etale-coh}).
Hence it is straightforward to see that the same conclusion of \cref{thm:B} holds for all odd $i$.
\end{remark}

The remark below says that the above even multiplicity type phenomenon has occurred in positive characteristic, even though there is no Hodge decomposition/symmetry.

\begin{remark}
\label{rmk:Deligne}
Let $X_0$ be a smooth projective variety over a finite field $\bF_q$ of characteristic $p$,
and $X = X_0 \times_{\bF_q} \ol\bF_q$ the base change of $X_0$ to an algebraic closure $\ol\bF_q$ of $\bF_q$.
Let $\Frob$ denote the Frobenius endomorphism of $X$ relative to $\bF_q$.
Deligne proved that the characteristic polynomial $\Phi_i(t)$ of the pullback $\Frob^*$ on $H^i_{\et}(X, \bQ_\ell)$ has integer coefficients independent of $\ell$, and all of its roots are of modulus $q^{i/2}$ (see \cite[Th\'eor\`eme~1.6]{Deligne74}).
Later, building on his earlier work, he also proved the hard Lefschetz theorem for \'etale cohomology (see \cite[Th\'eor\`eme~4.1.1]{Deligne80}).
Then combining with Poincar\'e duality, there is a nondegenerate pairing
\[
H^i_{\et}(X, \bQ_\ell) \times H^i_{\et}(X, \bQ_\ell) \to \bQ_\ell(-i),
\]
which is compatible with the Frobenius action.
When $i$ is odd, the pairing is alternating and hence the $i$-th Betti number $b_i(X)$ is even (see \cite[Corollaire~4.1.5]{Deligne80}).
Putting all together, we see that the $b_i(X)$ complex roots of $\Phi_i(t)$ fall into $b_i(X)/2$ pairs with each one consisting of two conjugate complex numbers and all real roots of $\Phi_i(t)$ are of even multiplicity (see also \cite{Suh12,EJ15,SZ16}).
One may ask if an analog of this even multiplicity type result holds for arbitrary correspondences.
It turns out to follow from the standard conjectures, as kindly pointed out to the author by Deligne.
\end{remark}

For an arbitrary smooth projective variety $X$ over $\bk$, applying \cref{thm:B} to its Albanese variety $\Alb(X)$ which is an abelian variety of dimension equal to the geometric irregularity $q(X)\coloneqq b_1(X)/2$, we have the following direct corollary.

\begin{corollary}
\label{cor:A}
Let $f \colon X \to X$ be a self-morphism of a smooth projective variety $X$ of dimension $n$ over $\bk$.
Let $P_i(f,t)$ denote the characteristic polynomial of the pullback $f^*$ on $H^i_{\emph\et}(X, \bQ_\ell)$.
Then for $i=1$ and $2n-1$, the $b_i(X)$ complex roots of $P_i(f,t)$ fall into $b_i(X)/2$ pairs with each one consisting of two conjugate complex numbers and all real roots of $P_i(f,t)$ are of even multiplicity.
\end{corollary}

Since our proof of \cref{thm:B} relies on the classification of the endomorphism $\bQ$-algebras of simple abelian varieties by Albert in the 1930s (see \cite[\S 18.2]{Oort08} and references therein), we introduce the following notion.

\begin{definition}
\label{def:Albert}
Let $\alpha\colon X\to X$ be an endomorphism of an abelian variety $X$ of dimension $g$ over $\bk$.
Let $P_\alpha(t)$ denote the characteristic polynomial of the pullback $\alpha^*$ on $H^1_{\et}(X, \bQ_\ell)$.
A (monic) complex polynomial $P_\alpha^{\Albert}(t) \in \bC[t]$ is called an {\it Albert polynomial of $\alpha$}, if $P_\alpha(t) = P_\alpha^{\Albert}(t)  \overline{P_\alpha^{\Albert}(t)}$.
\end{definition}

We have seen in \cref{thm:B} that there always exists {\it an} Albert polynomial $P_\alpha^{\Albert}(t)$ of $\alpha$;
however, it may not be unique because we are free to switch those conjugate complex roots.

Note that when $X$ is a complex abelian variety, there is a canonical choice of an Albert polynomial of an endomorphism $\alpha$, namely, the characteristic polynomial of the pullback $\alpha^*$ on the Dolbeault cohomology group $H^{1,0}(X, \bC)$; this is also equal to the characteristic polynomial of the analytic representation $\rho_a(\alpha)$ of $\alpha$, where $\rho_a(\alpha)$ is the induced linear map of $\alpha$ on the universal cover $\bC^g$ of $X$.
So the notion of Albert polynomial could be regarded as a characteristic-free substitute of the characteristic polynomial of the analytic representation.

\begin{remark}
\label{rmk:overQ}
In the category of abelian varieties, it is inevitable to work on endomorphisms with rational coefficients, namely, the endomorphism $\bQ$-algebra $\End^0(X) \coloneqq \End(X) \otimes_\bZ \bQ$.
For any $\alpha \in \End^0(X)$ with $k\alpha \in \End(X)$ for some positive integer $k$, the {\it characteristic polynomial} $P_\alpha(t)$ of $\alpha$ is defined as $k^{-2g} P_{k\alpha}(kt) \in \bQ[t]$, which is still monic (see \cite[Proposition~12.4]{Milne86}).
Hence if $P_{k\alpha}^{\Albert}(t)$ is an Albert polynomial of $k\alpha \in \End(X)$, we call $k^{-g} P_{k\alpha}^{\Albert}(kt) \in \bC[t]$ {\it an Albert polynomial} $P_\alpha^{\Albert}(t)$ of $\alpha \in \End^0(X)$.
Clearly, this is independent of the choice of $k$.
\end{remark}

We henceforth fix a polarization $\phi = \phi_\sL \colon X \to \what X$, where $\sL \coloneqq \sO_X(H_X)$ is an ample line bundle on $X$ associated to a fixed ample divisor $H_X$ on $X$.
Let
\[
\alpha^\dagger \coloneqq \phi^{-1}\circ \what \alpha \circ \phi \in \End^0(X)
\]
be the {\it Rosati involution of $\alpha$} (see \cite[\S20 and \S21]{Mumford} and \cite[\S17]{Milne86} for more details).
It would follow from \cref{lemma:unique} that for a symmetric element $\alpha \in \End^0(X)$, i.e., $\alpha^\dagger = \alpha$, its Albert polynomial $P_\alpha^{\Albert}(t)$, in the sense of \cref{rmk:overQ}, is unique and lies in $\bR[t]$.
We now give a geometric characterization of the coefficients of {\it the} Albert polynomial of $\alpha^\dagger \circ \alpha$.

\begin{corollary}
\label{cor:B}
Let $\alpha \colon X \to X$ be an endomorphism of an abelian variety $X$ of dimension $g$ over $\bk$.
Then the Albert polynomial $P_{\alpha^\dagger \circ \alpha}^{\Albert}(t)$ of $\alpha^\dagger \circ \alpha$ is unique and has rational coefficients.
Moreover, if we write
\[
P_{\alpha^\dagger \circ \alpha}^{\Albert}(t) = \sum_{k=0}^g (-1)^k c_k \, t^{g-k},
\]
then for any $0\le k \le g$, one has
\[
c_k = \binom{g}{k} \frac{\, \alpha^* H_X^k \cdot H_X^{g-k}\, }{H_X^g} \in \bQ,
\]
where both $\alpha^* H_X^k \cdot H_X^{g-k}, \, H_X^g \in \bQ$ are intersection numbers.
\end{corollary}

The above result generalizes \cite[\S21, Theorem~1]{Mumford}, where $k = 1$, to more general intersection products.
One might be able to use the exterior product calculation, as in his proof, to deduce the formula.
However, our proof is geometric in nature and essentially relies on the positivity of the Rosati involution.

\subsection{An application to algebraic dynamics}
\label{subsec:app}

The second main result is an application of Albert polynomials to the study of certain dynamical problems on abelian varieties.
Indeed, our \cref{thm:A} confirms Truong's dynamical degree comparison \cref{conj:DDC} for abelian varieties, which we shall state in detail now.

Let $f\colon X \ratmap X$ be a dominant rational self-map of a smooth projective variety $X$ of dimension $n$ over $\bk$.
In order to measure the dynamical complexity of $f$ under iterations, one can associate to this $f$ two dynamical invariants as follows.
Let $\ell$ be a prime different from the characteristic of $\bk$.
Fix a field isomorphism $\iota\colon \ol\bQ_\ell \isom \bC$ so that we may speak of the complex absolute value of an element of $\ol\bQ_\ell$:
\[
|a|_{\iota} \coloneqq |\iota(a)|, \textrm{ for any } a \in \ol\bQ_\ell.
\]
We endow a norm $\norm{\cdot}_{\iota}$ on the finite-dimensional $\bQ_\ell$-vector space $H^{\bullet}_{\et}(X, \bQ_\ell)$.
For any $0\le i\le 2n$, the {\it $i$-{th} cohomological dynamical degree $\chi_i(f)_{\iota}$ of $f$} (with respect to $\iota$) is defined by
\[
\chi_i(f)_{\iota} \coloneqq \limsup_{m\to \infty} \big\|(f^m)^*|_{H^i_{\et}(X, \bQ_\ell)}\big\|^{1/m}_{\iota},
\]
where the pullback is always defined in the sense of correspondences (see, e.g., \cite[\S 1.3]{Kleiman68}).

One can define another dynamical degree of $f$ using algebraic cycles.
Indeed, let $\N^k(X)$ denote the group of algebraic cycles of codimension $k$ on $X$ modulo numerical equivalence.
Note that $\N^k(X)$ is a finitely generated free abelian group (see, e.g., \cite[Theorem~3.5]{Kleiman68}).
For any $0\le k\le n$, we define the {\it $k$-{th} numerical dynamical degree $\lambda_k(f)$ of $f$} as
\[
\lambda_k(f) \coloneqq \lim_{m\to \infty} \big\|(f^m)^*|_{\N^k(X)_\bR}\big\|^{1/m},
\]
where we fix an arbitrary norm $\norm{\cdot}$ on the $\bR$-vector space $\N^k(X)_\bR \coloneqq \N^k(X) \otimes_\bZ \bR$.
The above limit defining $\lambda_k(f)$ exists by \cite{Truong20,Dang20}.
Note that when $f$ is a self-morphism, the pullback action is compatible with the iteration, that is, $(f^m)^*=(f^*)^m$; it thus follows that $\chi_i(f)_{\iota} = \rho(f^*|_{H^i_{\et}(X, \bQ_\ell)})_{\iota}$ and $\lambda_k(f) = \rho(f^*|_{\N^k(X)_\bR})$, where $\rho(\varphi)$ denotes the spectral radius of certain linear operator $\varphi$.

Inspired by results in complex dynamics (see \cite{DS17} for a survey), in Esnault--Srinivas \cite{ES13}, and Weil's Riemann hypothesis (now Deligne's theorem \cite{Deligne74}), Truong proposed the following dynamical degree comparison conjecture (in fact, his version is formulated for more general dynamical correspondences which are natural generalizations of dominant rational maps).

\begin{conjecture}[{cf.~\cite[Question~2]{Truong}}]
\label{conj:DDC}
Let $f\colon X\ratmap X$ be a dominant rational self-map of a smooth projective variety $X$ of dimension $n$ over $\bk$.
Then for any $0\le k\le n$, we have $\chi_{2k}(f)_{\iota} = \lambda_k(f)$.
\end{conjecture}

As an application of Albert polynomials (see \cref{def:Albert,rmk:overQ,cor:B}), we give an affirmative answer to \cref{conj:DDC} for surjective self-morphisms of abelian varieties, extending the main result of \cite{Hu19} from the case $k=1$ to the general case.

\begin{theorem}
\label{thm:A}
Let $f$ be a surjective self-morphism of an abelian variety $X$ of dimension $g$ over $\bk$.
Then for any $0\le k\le g$, we have $\chi_{2k}(f)_{\iota} = \lambda_k(f)$, or equivalently,
\[
\rho(f^*|_{H^{2k}_{\emph\et}(X, \bQ_\ell)})_{\iota} = \rho(f^*|_{\N^k(X)_\bR}).
\]
\end{theorem}

We conclude this section with some remarks on \cref{conj:DDC,thm:A}.

\begin{remark}
\label{DDC}
We first discuss some related work and known results on \cref{conj:DDC}.
\begin{enumerate}[(1)]
\item When $\bk \subseteq \bC$, we may associate to $(X, f)$ a projective (and hence compact K\"ahler) manifold $X_\bC$ and a dominant meromorphic self-map $f_\bC$. Then by Artin's comparison theorem and Hodge theory, it is not hard to show that $\chi_{2k}(f)_{\iota} = \lambda_k(f)$ (see, e.g., \cref{lemma:dg-overC}); both of them agree with the usual dynamical degree defined using the Dolbeault cohomology group $H^{k,k}(X_\bC, \bC)$ in the context of complex dynamics (see \cite[\S4]{DS17}).

\item For an arbitrary algebraically closed field (in particular, of positive characteristic), Esnault and Srinivas \cite{ES13} proved that for an automorphism of a smooth projective surface, the second cohomological dynamical degree coincides with the first numerical dynamical degree.
However, as far as we know, \cref{conj:DDC} is still open even for polarized endomorphisms of $K3$ surfaces.
Recently, in joint work with Truong \cite{HT}, we verify this for polarized endomorphisms of Kummer surfaces using dynamical correspondences.

\item It turns out that \cref{conj:DDC} is related to Weil's Riemann hypothesis.
Precisely, as we have seen in \cref{rmk:Deligne}, if $X = X_0 \times_{\bF_q} \ol\bF_q$ is the base change of a smooth projective variety $X_0$ over a finite field $\bF_q$ to $\ol\bF_q$ and $\Frob$ is the Frobenius endomorphism of $X$ relative to $\bF_q$, then Deligne's celebrated theorem asserts that all eigenvalues of the pullback $\Frob^*$ on $H^i_{\et}(X, \bQ_\ell)$ are algebraic integers of modulus $q^{i/2}$.
In particular, we have that $\chi_i(\Frob)_{\iota} = q^{i/2}$ for all $i$ and $\iota$.
On the other hand, it is easy to see that the $k$-{th} numerical dynamical degree $\lambda_k(\Frob)$ of $\Frob$ is equal to $q^k$ (see \cref{thm:P} for a more general treatment of polarized endomorphisms).
It follows that $\chi_{2k}(\Frob)_{\iota} = \lambda_k(\Frob)$ for all $k$.

Conversely, using a standard product trick (i.e., consider the product morphism of the product variety), the K\"unneth formula, and Poincar\'e duality, one can show that \cref{conj:DDC} (for $X\times X$) implies Weil's Riemann hypothesis (for $X$; see \cite[Lemma~4.4]{HT}).
Recently, joint with Truong \cite{HT}, we systematically investigate the relations between various comparison conjectures on correspondences, Weil's Riemann hypothesis, and the standard conjectures.

\item When $f$ is a surjective self-morphism of $X$, Truong \cite[Theorem~1.1]{Truong} proved a weaker statement that
\[
h_{\et}(f) \coloneqq \max_{0\le i\le 2n} \log \chi_i(f)_{\iota} = \max_{0\le k\le n} \log \lambda_k(f) \eqqcolon h_{\alg}(f),
\]
which asserts that the (\'etale) entropy $h_{\et}(f)$ coincides with the algebraic entropy $h_{\alg}(f)$ in the sense of \cite[\S6.3]{ES13}.
As a consequence, the spectral radius of the pullback $f^*$ on the even-degree cohomology $H^{2\bullet}_{\et}(X, \bQ_\ell)$ coincides with the spectral radius of $f^*$ on the total cohomology $H^{\bullet}_{\et}(X, \bQ_\ell)$.
See also \cite{Shuddhodan19} for a different approach towards this equality using dynamical zeta functions.
In particular, when $\dim X = 2$, if one has $\lambda_1(f) \ge \lambda_2(f)$ (which holds for automorphisms of surfaces as considered in \cite{ES13}), then $\chi_2(f)_{\iota} \le \lambda_1(f)$ and hence $\chi_2(f)_{\iota} = \lambda_1(f)$; see \cite[Theorem~1.4]{Truong}.
Note that when $\bk \subseteq \bC$, by the fundamental work of Gromov \cite{Gromov03} and Yomdin \cite{Yomdin87}, the algebraic entropy also equals the topological entropy $h_{\textrm{top}}(f_\bC)$ of the holomorphic dynamical system $(X_\bC, f_\bC)$; see \cite[\S4]{DS17} for details.
\end{enumerate}
\end{remark}

\begin{remark}
\label{rmk:A1}
\begin{enumerate}[(1)]
\item Note that our \cref{thm:A} implies Weil's Riemann hypothesis for abelian varieties (see \cref{DDC}(3) and \cref{rmk:A2}).
Hence borrowing the ideas from dynamical systems, we extend a classical result of Weil \cite[\S21, Application~II]{Mumford}.
This also answers a \href{http://aimpl.org/cohomabelian/1/}{question} of Esnault raised in the AIM workshop \emph{\href{https://aimath.org/pastworkshops/cohomabelian.html}{Cohomological methods in abelian varieties}}, 2012.
For an automorphism $f$ of an abelian variety $X$, she actually asked whether the spectral radius of $f^*|_{H^{\bullet}_{\et}(X, \bQ_\ell)}$ is achieved on the subring in $H^{\bullet}_{\et}(X, \bQ_\ell)$ generated by all $(f^m)^*H_X$ with $m\in \bZ$.
Our \cref{thm:B} first asserts that $\rho(f^*|_{H^{\bullet}_{\et}(X, \bQ_\ell)})_{\iota} = \rho(f^*|_{H^{2k}_{\et}(X, \bQ_\ell)})_{\iota}$ for certain $0<k<g$ and then \cref{thm:A} applies.

\item As mentioned before, when $f$ is an automorphism of an abelian surface, \cref{thm:A} is already known by Esnault and Srinivas (see \cite[\S4]{ES13}).
Even in this two-dimensional case, their proof is quite involved and different from ours (actually, after the standard spreading out and specialization argument, they invoke a celebrated theorem of Tate \cite{Tate66}).
Previously, the author \cite{Hu19} has dealt with the case of abelian varieties but only when $k = 1$.
In that case, we can naturally embed the N\'eron--Severi group $\NS(X)$ of an abelian variety $X$ into the semisimple $\bQ$-algebra $\End^0(X)$ of endomorphisms (with $\bQ$-coefficients) and extend the pullback $f^*$ on $\NS(X)$ to $\End^0(X)$.
Thanks to the structure theorem of $\End^0(X)$, the later action $f^*$ on $\End^0(X)$ can be reinterpreted in terms of matrices related to the matrix representing the pullback $f^*$ on $H^1_{\et}(X, \bQ_\ell)$.
However, since algebraic cycles of higher codimensions are much more subtle \cite{OSullivan11}, we are not aware of any embedding of $\N^k(X)_\bR$ into certain nice spaces for $k\ge 2$, on which the pullback $f^*$ can be extended and easily understood.
The proof for general $k$ presented in this paper is different from \cite{Hu19}.

\item If an endomorphism $\alpha\colon X\to X$ is not surjective, one can also proceed by replacing $X$ (possibly several times) by the image $Y\coloneqq \alpha(X)$, which is then a lower-dimensional abelian subvariety of $X$ such that $\alpha|_Y\colon Y\to Y$ is surjective and $X$ is isogenous to $Y\times Z$ for some abelian subvariety $Z$ of $X$.
\end{enumerate}
\end{remark}


\section{Preliminaries on abelian varieties}
\label{section:prelim}


We refer to \cite{Mumford} and \cite{Milne86} for standard notations and terminologies on abelian varieties.

\begin{notation*}
The following notations remain in force throughout the rest of this paper unless otherwise stated.
\renewcommand*{\arraystretch}{1.1}
\begin{longtable*}{p{2cm} p{12cm}}
$\bk$ & an algebraically closed field of arbitrary characteristic \\
$\ell$ & a prime different from the characteristic of $\bk$ \\
$X$ & an abelian variety of dimension $g$ over $\bk$ \\
$\what X$ & the dual abelian variety $\Pic^0(X)$ of $X$ \\
$\End^0(X)$ & $\End(X) \otimes_\bZ \bQ$, the endomorphism $\bQ$-algebra of $X$ \\
$\alpha$ & an endomorphism of $X$, or an element of $\End^0(X)$ \\
$\N^k(X)_\bR$ & $\N^k(X) \otimes_\bZ \bR$, the $\bR$-vector space of numerical equivalence classes of codimension-$k$ real algebraic cycles on $X$ \\
$H^i_{\et}(X, \bQ_\ell)$ & $H^i_{\et}(X, \bZ_\ell) \otimes_{\bZ_\ell} \bQ_\ell$, the $\ell$-adic \'etale cohomology group of degree $i$ \\
$\lambda_k(\alpha)$ & the $k$-th numerical dynamical degree of $\alpha$ \\
$\chi_i(\alpha)_{\iota}$ & the $i$-th cohomological dynamical degree of $\alpha$ with respect to $\iota\colon \ol\bQ_\ell \isom \bC$ \\
$P_\alpha(t)$ & the characteristic polynomial of $\alpha$, which has degree $2g$ and integer (resp. rational) coefficients if $\alpha \in \End(X)$ (resp. if $\alpha \in \End^0(X)$) \\
$\chi_\alpha^{\reduced}(t)$ & the reduced characteristic polynomial of $\alpha$ \\
$P_\alpha^{\Albert}(t)$ & an Albert polynomial of $\alpha$, which has degree $g$ and complex coefficients \\
$\phi_\sM$ & the induced homomorphism of a line bundle $\sM$ on $X$: \\
 & \quad \quad $\phi_\sM \colon X \to \what X, \ \ x \mapsto t_x^*\sM \otimes \sM^{-1}$ \\
$\phi = \phi_{\sL}$ & the fixed polarization of $X$ induced from some fixed ample line bundle $\sL = \sO_X(H_X)$ \\
$^\dagger$ & the Rosati involution on $\End^0(X)$ defined in the following way: \\
 & \quad \quad $\alpha \mapsto \alpha^\dagger \coloneqq \phi^{-1}\circ \what \alpha \circ \phi$, for any $\alpha \in \End^0(X)$ \\
$\bH$ & the standard quaternion algebra over $\bR$ whose basis is $\{ 1, \bbi, \bbj, \bbk\}$
\end{longtable*}
\end{notation*}

For the convenience of the reader, we include several important structure theorems on the \'etale cohomology groups, the endomorphism algebras, and the N\'eron--Severi groups of abelian varieties.
We refer to \cite[\S19-21]{Mumford} and \cite{Milne86} for details and to \cite{Tate66,Oort08} for more on abelian varieties over finite fields.

First, the \'etale cohomology groups of abelian varieties are simple to describe.

\begin{theorem}[{cf.~\cite[Theorem~15.1]{Milne86}}]
\label{thm:etale-coh}
Let $X$ be an abelian variety of dimension $g$ over $\bk$, and let $\ell$ be a prime different from the characteristic of $\bk$. Let $T_\ell(X) \coloneqq \varprojlim_{n} X(\bk)[\ell^n]$ be the Tate module of $X$, which is a free $\bZ_\ell$-module of rank $2g$.
\begin{itemize}
\item[(a)] There is a canonical isomorphism
\[
H^1_{\emph\et}(X, \bZ_\ell) \isom T_\ell(X)^\vee \coloneqq \Hom_{\bZ_\ell}(T_\ell(X), \bZ_\ell).
\]
\item[(b)] The cup-product pairing induces isomorphisms
\[
H^i_{\emph\et}(X, \bZ_\ell) \isom \bigwedge\nolimits^{\! i} H^1_{\emph\et}(X, \bZ_\ell),
\]
for all $i$.
In particular, $H^i_{\emph\et}(X, \bZ_\ell)$ is a free $\bZ_\ell$-module of rank $\binom{2g}{i}$.
\end{itemize}
\end{theorem}

Furthermore, the functor $T_\ell$ induces an $\ell$-adic representation of the endomorphism algebra. In general, we have:

\begin{theorem}[{cf.~\cite[\S19, Theorem~3]{Mumford}}]
\label{thm:l-adic-rep}
For any two abelian varieties $X$ and $Y$, the group $\Hom(X,Y)$ of homomorphisms of $X$ into $Y$ is a finitely generated free abelian group, and the natural homomorphism of $\bZ_\ell$-modules
\[
\Hom(X, Y) \otimes_\bZ \bZ_\ell \to \Hom_{\bZ_\ell}(T_\ell(X), T_\ell(Y))
\]
induced by $T_\ell \colon \Hom(X, Y) \to \Hom_{\bZ_\ell}(T_\ell(X), T_\ell(Y))$ is injective.
\end{theorem}

For a homomorphism $f \colon X \to Y$ of abelian varieties, its {\it degree} $\deg f$ is defined to be the order of the kernel $\ker f$, if it is finite, and $0$ otherwise.
In particular, the degree of an isogeny is always a positive integer.
We can extend this notion to any $\alpha \in \End^0(X)$ by setting $\deg \alpha = n^{-2g} \deg(n\alpha)$ if $n\alpha \in \End(X)$.

\begin{theorem}[{cf.~\cite[\S19, Theorem~4]{Mumford} and \cite[\S12, Proposition~12.4]{Milne86}}]
\label{thm:char-poly}
For any $\alpha \in \End^0(X)$, there exists a monic polynomial $P_\alpha(t) \in \bQ[t]$ of degree $2g$ such that $P_\alpha(r) = \deg(r_X - \alpha)$ for all rational numbers $r$.
Moreover, if $\alpha\in \End(X)$, then $P_\alpha(t)$ has only integer coefficients; it is also equal to the characteristic polynomial $\det(t \, \bI_{2g} - V_\ell(\alpha))$ of the induced linear map $V_\ell(\alpha)$ of $\alpha$ on the Tate space $V_\ell(X) \coloneqq T_\ell(X) \otimes_{\bZ_\ell} \bQ_\ell$.
\end{theorem}

We call $P_\alpha(t)$ as in \cref{thm:char-poly} the {\it characteristic polynomial of $\alpha$}.
On the other hand, we can assign to each $\alpha$ the characteristic polynomial $\chi_\alpha(t)$ of $\alpha$ as an element of the semisimple $\bQ$-algebra $\End^0(X)$.
Namely, we define $\chi_\alpha(t)$ to be the characteristic polynomial of the left multiplication $\alpha_L \colon \beta \mapsto \alpha \circ \beta$ for $\beta \in \End^0(X)$ which is a $\bQ$-linear transformation on $\End^0(X)$.
Note that the above definition of $\chi_\alpha(t)$ makes no use of the fact that $\End^0(X)$ is semisimple.
Actually, for semisimple $\bQ$-algebras, it is much more useful to consider the so-called reduced characteristic polynomials.

We recall some basic definitions on semisimple algebras (see \cite[\S9]{Reiner03} for details).

\begin{definition}
\label{def:red-char}
Let $R$ be a finite-dimensional semisimple algebra over a field $F$ of characteristic zero, and write
\[
R = \bigoplus_{i=1}^{k} R_i,
\]
where each $R_i$ is a simple $F$-algebra.
For any element $r \in R$, as above, we denote by $\chi_r(t)$ the {\it characteristic polynomial of $r$}.
Namely, $\chi_r(t)$ is the characteristic polynomial of the left multiplication $r_L \colon s \mapsto rs$ for $s\in R$.
Let $K_i$ be the center of $R_i$.
Then there exists a finite field extension $E_i/K_i$ splitting $R_i$ (see \cite[\S7b]{Reiner03}), i.e., we have
\[
h_i \colon R_i \otimes_{K_i} E_i \xrightarrow{~\sim~} \Mat_{d_i}(E_i), \quad \text{where } [R_i:K_i] = d_i^2.
\]
Write $r = r_1 + \cdots + r_k$ with each $r_i \in R_i$.
We first define the {\it reduced characteristic polynomial $\chi_{r_i}^{\reduced}(t)$ of each $r_i$} as follows (see \cite[Definition~9.13]{Reiner03}):
\[
\chi_{r_i}^{\reduced}(t) \coloneqq \Nm_{K_i/F} (\det(t \, \bI_{d_i} - h_i(r_i \otimes_{K_i} \! 1_{E_i})) ) \in F[t].
\]
Note that $\det(t \, \bI_{d_i} - h_i(r_i \otimes_{K_i} \! 1_{E_i}))$ lies in $K_i[t]$, and is independent of the choice of the splitting field $E_i$ of $R_i$ (see, e.g., \cite[Theorem~9.3]{Reiner03}).
The {\it reduced norm of $r_i$} is defined by
\[
\Nm_{R_i/F}^{\reduced}(r_i) \coloneqq \Nm_{K_i/F} (\det (h_i(r_i \otimes_{K_i} \! 1_{E_i})) ) \in F.
\]
Finally, as one expects, the {\it reduced characteristic polynomial $\chi_{r}^{\reduced}(t)$} and the {\it reduced norm $\Nm_{R/F}^{\reduced}(r)$ of $r=\sum_i r_i$} are defined by the products:
\begin{equation*}
\label{eq:def-red-char}
\chi_{r}^{\reduced}(t) \coloneqq \prod_{i=1}^{k} \chi_{r_i}^{\reduced}(t) \quad \text{and} \quad 
\Nm_{R/F}^{\reduced}(r) \coloneqq \prod_{i=1}^{k} \Nm_{R_i/F}^{\reduced}(r_i).
\end{equation*}
\end{definition}

\begin{remark}
\label{rmk:red-char}
\begin{enumerate}[(1)]
\item It follows from \cite[Theorem~9.14]{Reiner03} that
\begin{equation}
\label{eq:Rei03-Thm9.14}
\chi_r(t) = \prod_{i=1}^{k} \chi_{r_i}(t) = \prod_{i=1}^{k} \chi_{r_i}^{\reduced}(t)^{d_i}.
\end{equation}

\item Note that reduced characteristic polynomials and norms are not affected by change of ground field (see, e.g., \cite[Theorem~9.27]{Reiner03}).
\end{enumerate}
\end{remark}

We now apply the above algebraic set-up to $R = \End^0(X)$, the semisimple $\bQ$-algebra of endomorphisms (with $\bQ$-coefficients) of an abelian variety $X$ over $\bk$.
For any $\alpha \in \End^0(X)$, let $\chi_{\alpha}^{\reduced}(t) \in \bQ[t]$ denote its {\it reduced characteristic polynomial}.
For simplicity, let us first consider the case when
\[
X = A^n
\]
is a power of a simple abelian variety $A$ over $\bk$.
Hence $\End^0(X) \isom \Mat_n(D)$, where $D\coloneqq \End^0(A)$ is a division ring.
Let $K$ denote the center of $D$ which is a field, and $K_0$ the maximal totally real subfield of $K$. Set
\begin{equation}
\label{eq:ed}
d^2 = [D:K], \ e = [K:\bQ], \ \text{ and } \ e_0 = [K_0:\bQ].
\end{equation}
Then the equality \eqref{eq:Rei03-Thm9.14} reads as
\[
\chi_\alpha(t) = \chi_{\alpha}^{\reduced}(t)^{dn}.
\]

Let $V_1, \ldots, V_e$ denote the $e$ nonisomorphic irreducible representations of $\End^0(X)$ over $\ol{\bQ}$, where each one has degree $dn$.
Note that for any $\alpha\in \End^0(X)$, the reduced characteristic polynomial $\chi_\alpha^{\reduced}(t)$ defined above is exactly the same as the characteristic polynomial of $\alpha$ acting on $\bigoplus_i V_i$.
We thus call $V^{\reduced} \coloneqq \bigoplus_i V_i$ the reduced representation of $\End^0(X)$.
The proposition below shows that the two characteristic polynomials $P_\alpha(t)$ and $\chi_{\alpha}^{\reduced}(t)$ are closely related.

\begin{proposition}[{cf.~\cite[Proposition~12.12]{Milne86}}]
\label{prop:Milne}
With notation as above, the representation $T_\ell(X) \otimes_{\bZ_\ell} \ol{\bQ}_\ell = V_\ell(X) \otimes_{\bQ_\ell} \ol{\bQ}_\ell$ of $\End^0(X)$ induced from \cref{thm:l-adic-rep} is isomorphic to a direct sum of $m \coloneqq 2g/(edn)$ copies of $V^{\reduced} \otimes_{\ol{\bQ}} \ol{\bQ}_\ell$.
In particular, for any $\alpha\in \End^0(X)$, we have
\[
\bQ[t] \ni P_\alpha(t) = \chi_{\alpha}^{\reduced}(t)^m.
\]
Moreover, if $\alpha\in \End(X)$, then by Gauss's primitivity lemma, $\chi_{\alpha}^{\reduced}(t) \in \bZ[t]$.
\end{proposition}

It is not hard to extend \cref{prop:Milne} to the case when $X$ is an arbitrary abelian variety over $\bk$ (which is not necessarily of the form $A^n$).
Indeed, by Poincar\'e's complete reducibility theorem (see, e.g., \cite[\S19, Theorem~1]{Mumford}), $X$ is isogenous to a product $X_1 \times \cdots \times X_s$, where the $X_j = A_j^{n_j}$ are powers of mutually non-isogenous simple abelian varieties $A_j$, and
\[
\End^0(X) \isom \prod_{j=1}^{s} \End^0(X_j) \isom \prod_{j=1}^{s} \Mat_{n_j}(\End^0(A_j)).
\]
Write the image of $\alpha\in \End^0(X)$ as the product $\alpha_1 \times \cdots \times \alpha_s$ with each $\alpha_j \in \End^0(X_j)$.
Let $V_{j,1}, \ldots, V_{j,e_j}$ denote the $e_j$ nonisomorphic irreducible representations of $\End^0(X_j)$ over $\ol{\bQ}$, where each one has degree $d_jn_j$.
Denote by $V_j^{\reduced} \coloneqq V_{j,1} \oplus \cdots \oplus V_{j,e_j}$ the reduced representation of $\End^0(X_j)$.

\begin{proposition}
\label{prop:MilneII}
With notation as above, the representation $T_\ell(X) \otimes_{\bZ_\ell} \ol{\bQ}_\ell = V_\ell(X) \otimes_{\bQ_\ell} \ol{\bQ}_\ell$ of $\End^0(X)$ is isomorphic to a direct sum of the $m_j$ copies of $V_j^{\reduced} \otimes_{\ol{\bQ}} \ol{\bQ}_\ell$ as follows:
\[
V_\ell(X) \otimes_{\bQ_\ell} \ol{\bQ}_\ell \isom \bigoplus_{j=1}^s \, (V_j^{\reduced})^{\oplus m_j} \otimes_{\ol{\bQ}} \ol{\bQ}_\ell,
\]
where $m_j = 2\dim X_j /(e_jd_jn_j)$ is a positive integer depending only on $X_j$ (and $\End^0(X_j)$; see \cref{prop:Milne}).
In particular, for any $\alpha\in \End^0(X)$, we have
\[
\bQ[t] \ni P_\alpha(t) = \prod_{j=1}^s P_{\alpha_j}(t) = \prod_{j=1}^s \chi_{\alpha_j}^{\reduced}(t)^{m_j}.
\]
Moreover, if $\alpha\in \End(X)$, then $\chi_{\alpha_j}^{\reduced}(t) \in \bZ[t]$ for all $j$.
\end{proposition}

We recall the following useful structure theorem on the endomorphism $\bR$-algebras of abelian varieties which would be used very frequently.

\begin{theorem}[{cf.~\cite[\S21, Theorems~2 and 6]{Mumford}}]
\label{thm:NS}
The endomorphism $\bR$-algebra $\End(X)_\bR \coloneqq \End(X) \otimes_\bZ \bR$ is isomorphic to a product of copies of $\Mat_r(\bR)$, $\Mat_r(\bC)$, or $\Mat_r(\bH)$.
Moreover, one can fix an isomorphism so that it carries the Rosati involution into the standard involution $\bA \mapsto \bA^* = \ol\bA^\sT$.
In particular, the real N\'eron--Severi space $\NS(X)_\bR \coloneqq \NS(X) \otimes_\bZ \bR$ is isomorphic to a product of Jordan algebras of the following types:
\begin{align*}
\sH_r(\bR) &= r \times r \text{ symmetric real matrices,} \\
\sH_r(\bC) &= r \times r \text{ Hermitian complex matrices,} \\
\sH_r(\bH) &= r \times r \text{ Hermitian quaternionic matrices.}
\end{align*}
\end{theorem}

For convenience, we include the following \Cref{the-table}, which gives the numerical constraints on the endomorphism $\bQ$-algebras of simple abelian varieties $A$; see, e.g., \cite[p.~202]{Mumford}.
Recall that $e$, $d$, and $e_0$ are positive integers defined in \cref{eq:ed}.

\begin{table}[!htbp]
\caption{Types of simple abelian varieties}
\label{the-table}
\renewcommand{\arraystretch}{1.2}
\begin{tabular}{|c|c|c|c|c| >{\centering\arraybackslash} m{230pt}|}
\hline \hline
\multirow{2}{*}{Type} & \multirow{2}{*}{$e$} & \multirow{2}{*}{$d$} & \multicolumn{2}{c|}{Restriction} & \multirow{2}{*}{$D = \End^0(A)$ with $g = \dim A$} \\
\cline{4-5}
& & & in char $0$ & in char $p$ & \\
\hline
\hline
I($e$) & $e_0$ & $1$ & $e|g$ & $e|g$ & $D = K = K_0$ is a totally real field \\
\hline
II($e$) & $e_0$ & $2$ & $2e|g$ & $2e|g$ & $D$ is a totally indefinite quaternion algebra over the totally real field $K = K_0$ \\
\hline
III($e$) & $e_0$ & $2$ & $2e | g$ & $e|g$ & $D$ is a totally definite quaternion algebra over the totally real field $K = K_0$ \\
\hline
IV($e_0, d$) & $2e_0$ & $d$ & $e_0d^2|g$ & $e_0d | g$ & $D$ is an (Albert) algebra over the CM-field $K \supset K_0$ \\
\hline
\hline
\end{tabular}
\end{table}


\section{Proof of Theorem \ref{thm:B}}
\label{section:proof-B}


\begin{proof}[Proof of \cref{thm:B}]
\label{proof:thm-B}
By Poincar\'e's complete reducibility theorem, it suffices to consider the case when $X$ is a power of a simple abelian variety $A$, say $X = A^n$, and $\alpha \in \End^0(X)$.
Indeed, with notation as in \cref{prop:MilneII}, suppose that for each $\alpha_j \in \End^0(X_j)$ we have found a complex polynomial $P_{\alpha_j}^{\Albert}(t)$ such that $P_{\alpha_j}(t) = P_{\alpha_j}^{\Albert}(t)  \overline{P_{\alpha_j}^{\Albert}(t)}$.
Note that here the characteristic polynomial $P_{\alpha_j}(t)$ of each $\alpha_j \in \End^0(X_j)$ is defined as $k_j^{-2\dim X_j}P_{k_j \alpha_j}(k_j t)$ if $k_j \alpha_j\in \End(X_j)$ (see \cite[Proposition~12.4]{Milne86}).
Then we simply take $\prod_{j=1}^s P_{\alpha_j}^{\Albert}(t)$ as our $P_\alpha^{\Albert}(t)$.
Now, after the above reduction, $\End^0(X)$ is isomorphic to the simple $\bQ$-algebra $\Mat_n(D)$ of all $n\times n$ matrices with entries in the division ring $D \coloneqq \End^0(A)$.
Let $K$ denote the center of $D$ which is a field, and $K_0$ the maximal totally real subfield of $K$.
As usual, we set
\[
d^2 = [D:K], \ e = [K:\bQ], \ \text{ and } \ e_0 = [K_0:\bQ].
\]
Let $\chi_{\alpha}^{\reduced}(t) \in \bQ[t]$ be the reduced characteristic polynomial of $\alpha$ (see \cref{def:red-char}).
Then by \cref{prop:Milne}, we have
\begin{equation}
\label{eq:char-poly}
P_\alpha(t) = \chi_{\alpha}^{\reduced}(t)^m,
\end{equation}
where $m = 2g / edn = 2 \dim A / ed$ is a positive integer.
We also note that
\[
\End(X)_\bR \coloneqq \End^0(X) \otimes_\bQ \bR \isom \Mat_n(D) \otimes_\bQ \bR \isom \Mat_n(D \otimes_\bQ \bR)
\]
is isomorphic to a product of copies of $\Mat_r(\bR)$, $\Mat_r(\bC)$, or $\Mat_r(\bH)$ (see \cref{thm:NS}).
According to Albert's classification of the endomorphism $\bQ$-algebras of simple abelian varieties (see, e.g., \cite[\S21, Theorem~2]{Mumford}), we have the following four cases.

\begin{case}
\label{case-I}
$D$ is of Type I$(e)$: $d=1$, $e=e_0$ and $D=K=K_0$ is a totally real algebraic number field and the involution on $D$ is the identity.
In this case, we have the following restriction:
\[
e \, | \dim A
\]
in any characteristic (see \Cref{the-table}).
It follows that the $m$ in the equation \eqref{eq:char-poly} is an even number.
Since $\chi_{\alpha}^{\reduced}(t)$ lies in $\bQ[t]$ by the definition of reduced characteristic polynomials, we simply take
\[
P_\alpha^{\Albert}(t) = \chi_{\alpha}^{\reduced}(t)^{m/2}.
\]
\end{case}

\begin{case}
\label{case-II}
$D$ is of Type II$(e)$: $d=2$, $e=e_0$, $K=K_0$ is a totally real algebraic number field and $D$ is an indefinite quaternion division algebra over $K$.
In this case, we have $2e \, | \dim A$ in any characteristic (see \Cref{the-table}) so that $m$ is also even.
We can take the same $P_\alpha^{\Albert}(t)$ as in \cref{case-I} because $\chi_{\alpha}^{\reduced}(t)$ has only rational coefficients.
\end{case}

\begin{case}
\label{case-III}
$D$ is of Type III$(e)$: $d=2$, $e=e_0$, $K=K_0$ is a totally real algebraic number field and $D$ is a definite quaternion division algebra over $K$.
In this case, we have $2e \, | \dim A$ in characteristic zero and $e \, | \dim A$ in positive characteristic (see \Cref{the-table}).
The characteristic zero case is exactly the same as before.
However, when the characteristic of the ground field $\bk$ is positive, the restriction $e \, | \dim A$ does not automatically guarantee the parity of $m$ as in the previous cases (e.g., it may happen that $e = \dim A$ and hence $m=1$; see \cref{eg:TypeIII}).
Nevertheless, we shall construct $P_{\alpha}^{\Albert}(t)$ using certain special property of quaternion algebras.

First, we have the following isomorphism
\[
\End(X)_\bR \isom \bigoplus_{i=1}^{e} \Mat_n(\bH),
\]
where $\bH \coloneqq \big(\frac{-1, \, -1}{\bR} \big)$ is the standard quaternion algebra over $\bR$.
Clearly, $\bH$ can be embedded, in a standard way (see, e.g., \cite[Example~9.4]{Reiner03}), into $\Mat_2(\bC) \isom \bH \otimes_\bR \bC$.
This would induce a natural embedding of $\Mat_n(\bH)$ into $\Mat_{2n}(\bC) \isom \Mat_n(\bH) \otimes_\bR \bC$ as follows (see, e.g., \cite[\S4]{Lee49}):
\[
\iota \colon \Mat_n(\bH) \longinjmap \Mat_{2n}(\bC) \quad \text{via} \quad \bA = \bA_1 + \bA_2 \, \bbj \longmapsto \iota (\bA) \coloneqq 
\begin{pmatrix}
\bA_1 & \bA_2 \\
-\overline{\bA}_2 & \overline{\bA}_1
\end{pmatrix}.
\]
In particular, a quaternionic matrix $\bA$ is Hermitian if and only if its image $\iota(\bA)$ is a Hermitian complex matrix.

Denote the image $\alpha \otimes_\bQ 1_\bR$ of $\alpha$ in $\End(X)_\bR$ by block diagonal matrix $\bA_\alpha = \bA_{\alpha, 1} \oplus \cdots \oplus \bA_{\alpha, e}$ with each $\bA_{\alpha, i} \in \Mat_n(\bH)$.
Also, we note that
\[
\End(X)_\bC \coloneqq \End(X)_\bR \otimes_\bR \bC \isom \bigoplus_{i=1}^{e} \Mat_{2n}(\bC),
\]
is a semisimple $\bC$-algebra with each summand $\Mat_{2n}(\bC)$ being a central simple $\bC$-algebra.
Then by \cref{def:red-char,rmk:red-char}, the reduced characteristic polynomial $\chi_{\alpha}^{\reduced}(t)$ of $\alpha$ is equal to the product of the characteristic polynomials $\det(t \, \bI_{2n} - \iota(\bA_{\alpha, i}))$ of $\iota(\bA_{\alpha, i})$.
Thanks to \cite[Theorem~5]{Lee49}, the $2n$ complex eigenvalues of each $\iota(\bA_{\alpha, i})$ fall into $n$ pairs with each pair consisting of two conjugate complex numbers;
regardless the multiplicity, denote them by $\pi_{i,1}, \ldots, \pi_{i,n}, \overline{\pi}_{i,1}, \ldots, \overline{\pi}_{i,n}$.
In fact, one can easily verify that if $\pi_{i,j} \in \bC$ is an eigenvalue of $\iota(\bA_{\alpha,i})$ so that
\[
\iota(\bA_{\alpha,i}) \begin{pmatrix}
\bu_{i,j} \\
\bv_{i,j}
\end{pmatrix}
= \pi_{i,j} \begin{pmatrix}
\bu_{i,j} \\
\bv_{i,j}
\end{pmatrix},
\text{ then }
\iota(\bA_{\alpha,i}) \begin{pmatrix}
-\ol\bv_{i,j} \\
\ol\bu_{i,j}
\end{pmatrix}
= \ol\pi_{i,j} \begin{pmatrix}
-\ol\bv_{i,j} \\
\ol\bu_{i,j}
\end{pmatrix},
\]
i.e., $\ol\pi_{i,j}$ is also an eigenvalue of $\iota(\bA_{\alpha,i})$ corresponding to the eigenvector $(-\ol\bv_{i,j}^\sT, \ol\bu_{i,j}^\sT)^\sT$.
This yields that
\[
\chi_{\alpha}^{\reduced}(t) = \prod_{i=1}^{e} \prod_{j=1}^{n} (t - \pi_{i,j})(t - \overline{\pi}_{i,j}).
\]

Clearly, by the equation \eqref{eq:char-poly},
\[
P_{\alpha}^{\Albert}(t) \coloneqq \prod_{i=1}^{e} \prod_{j=1}^{n} (t - \pi_{i,j})^m \in \bC[t]
\]
is what we want, though the choice of $\pi_{i,j}$ or $\ol\pi_{i,j}$ may not be canonical (see \cref{eg:TypeIII}).
\end{case}

\begin{case}
\label{case-IV}
$D$ is of Type IV$(e_0, d)$: $e=2e_0$ and $D$ is a division algebra over the CM-field $K \supsetneq K_0$
(i.e., $K$ is a totally imaginary quadratic extension of a totally real algebraic number field $K_0$).
In this case, neither the restriction $e_0d^2 \, | \dim A$ in characteristic zero nor the restriction $e_0d \, | \dim A$ in characteristic $p$ ensures the parity of the integer $m$ as in \cref{case-I,case-II}.
However, this last remaining case is also special enough so that the reduced characteristic polynomial of $\alpha$ is canonically equal to the product of a complex polynomial and its complex conjugate (bearing some similarity with \cref{case-III}).

In fact, in this case, the endomorphism $\bR$-algebra
\[
\End(X)_\bR \isom \bigoplus_{i=1}^{e_0} \Mat_{dn}(\bC),
\]
so that the image $\alpha \otimes_\bQ 1_\bR$ of $\alpha$ in $\End(X)_\bR$ could be represented by the block diagonal matrix $\bA_\alpha = \bA_{\alpha, 1} \oplus \cdots \oplus \bA_{\alpha, e_0}$ with each $\bA_{\alpha, i} \in \Mat_{dn}(\bC)$.
We now need to note that $\End(X)_\bR$ is a semisimple $\bR$-algebra, while the center of each component $\Mat_{dn}(\bC)$ is $\bC$.
Then by \cref{def:red-char,rmk:red-char}, the reduced characteristic polynomial $\chi_{\alpha}^{\reduced}(t)$ of $\alpha$ is equal to
\[
\prod_{i=1}^{e_0} \Nm_{\bC/\bR} (\det(t \, \bI_{dn} - \bA_{\alpha, i}) ) = \prod_{i=1}^{e_0} \det(t \, \bI_{dn} - \bA_{\alpha, i}) \cdot \det(t \, \bI_{dn} - \overline{\bA}_{\alpha, i}).
\]
By the equation \eqref{eq:char-poly} again, we now just take
\[
P_\alpha^{\Albert}(t) = \prod_{i=1}^{e_0} (\det(t \, \bI_{dn} - \bA_{\alpha, i}) )^m.
\]
\end{case}

We thus complete the proof of \cref{thm:B}.
\end{proof}

\begin{remark}
\label{rmk:B2}
It follows from the proof of \cref{thm:B} that if all factors $X_j$ of $X$ are of Type~I$(e)$ or II$(e)$ then the Albert polynomial $P_\alpha^{\Albert}(t)$ of a genuine endomorphism $\alpha \in \End(X)$ has integer coefficients.
In fact, to reduce potential inaccuracies, let us start with \cref{prop:MilneII} from the beginning.
So for each factor $X_j$ of Type I$(e)$ or II$(e)$, we have shown that
\[
P_{\alpha_j}^{\Albert}(t) = \chi_{\alpha_j}^{\reduced}(t)^{m_j/2} \in \bQ[t]
\]
is monic, where $m_j$ is an appropriate positive even integer depending only on $X_j$.
Now, the Albert polynomial $P_\alpha^{\Albert}(t)$ of $\alpha \in \End(X)$ constructed in \cref{thm:B} is just the product
\[
\prod_{j=1}^s P_{\alpha_j}^{\Albert}(t),
\]
so is a monic polynomial with rational coefficients.
Note, however, that $P_\alpha^{\Albert}(t)^2 = P_\alpha(t) \in \bZ[t]$.
This yields that all $P_{\alpha_j}^{\Albert}(t)$ and hence $P_\alpha^{\Albert}(t)$ itself have only integer coefficients either by Gauss's primitivity lemma or by noting that all roots are algebraic integers.
\end{remark}

\begin{remark}
\label{rmk:B3}
\begin{enumerate}[(1)]
\item Our construction of an Albert polynomial $P_\alpha^{\Albert}(t)$ is canonical in characteristic zero, but unfortunately, not canonical in positive characteristic once the endomorphism $\bQ$-algebra of Type~III$(e)$ occurs,
in which case our \cref{eg:TypeIII} of supersingular elliptic curves reveals that the canonical construction of an Albert polynomial seems not likely to exist.
This seems to be a big difference between characteristic zero and prime characteristic.
We also hope that there exists an intrinsic and classification-free construction.
Nevertheless, it turns out that for a symmetric endomorphism $\alpha$, its Albert polynomial is unique (see \cref{lemma:unique}).

\item It could happen that $P_\alpha(t)$ has no multiple roots at all.
Indeed, let $X$ be an abelian variety of dimension $g$ over a finite field $\bF_q$, and $\Frob$ the Frobenius endomorphism of $X$ relative to $\bF_q$.
Then according to \cite[\S3, Theorem~2]{Tate66}, the characteristic polynomial $P_{\,\Frob}(t)$ of $\Frob$ has no multiple root if and only if $\End^0(X) = \bQ[\Frob]$ is commutative.
For instance, we can choose $X$ as a simple CM abelian variety of Type IV$(g, 1)$, i.e., $\End^0(X)$ is a CM-field of degree $2g$ over $\bQ$.
See also \cite[Remark~5.2]{Oort88} for a discussion on the realization.
\end{enumerate}
\end{remark}

The example below indicates that in positive characteristic the canonical construction of an Albert polynomial does not seem likely to exist in \cref{case-III}.
Note that by \cite[Proposition~4.2]{Oort88}, if a simple abelian variety $A$ of dimension $g$ is of Type III$(g)$, it must be a supersingular elliptic curve, i.e., $g = 1$.
See also \cite[\S22, pp.~214--218]{Mumford} for a discussion about elliptic curves in positive characteristic.

\begin{example}
\label{eg:TypeIII}
Let $E$ be a supersingular elliptic curve over an algebraically closed field of positive characteristic, i.e., $\End^0(E)$ is of Type III$(1)$.
Equivalently, $\End(E) \otimes_\bZ \bR$ is the standard quaternion algebra $\bH$ over $\bR$, whose basis is $\{ 1, \bbi, \bbj, \bbk\}$.
Let $\alpha$ be an endomorphism of $E$ such that $\alpha \otimes_\bZ 1_\bR = a + b \, \bbi + c \, \bbj + d \, \bbk \in \bH$.
Clearly, the characteristic polynomial $P_\alpha(t)$ of $\alpha$ and the reduced characteristic polynomial $\chi_\alpha^{\reduced}(t)$ of $\alpha$ coincide.
Both of them are equal to
\[
t^2 - 2at + (a^2+b^2+c^2+d^2) = (t - \pi_1)(t - \ol\pi_1) \in \bR[t].
\]
Hence either $t - \pi_1$ or its complex conjugate is an Albert polynomial of $\alpha$ by \cref{def:Albert}.
In other words, there seems no way to obtain a canonical choice for $P_\alpha^{\Albert}(t)$ if $b^2+c^2+d^2 \ne 0$.
\end{example}


\section{Proof of Corollaries \ref{cor:A} and \ref{cor:B}}
\label{section:proof-cor}


\begin{proof}[Proof of \cref{cor:A}]
Let $\Alb(X)$ denote the Albanese variety of $X$, which is an abelian variety over $\bk$ of dimension equal to the geometric irregularity $q(X) \coloneqq b_1(X)/2$.
Let $\alb_X$ denote the Albanese morphism from $X$ to $\Alb(X)$.
Then there exists an induced morphism $a_f \colon \Alb(X) \to \Alb(X)$ such that $\alb_X$ is $f$-equivariant, i.e., $\alb_X \circ f = a_f \circ \alb_X$.
Note that the Albanese morphism $\alb_X$ also induces an isomorphism
\[
\alb_X^* \colon H^1_{\et}(\Alb(X), \bQ_\ell) \xrightarrow{~\sim~} H^1_{\et}(X, \bQ_\ell)
\]
between the first \'etale cohomology groups via the pullback (see, e.g., \cite[Theorem~2A9]{Kleiman68}).
We thus have the following commutative diagram of isomorphic $\bQ_\ell$-vector spaces:
\[
\xymatrix{
H^1_{\et}(\Alb(X), \bQ_\ell) \ar[rr]^{a_f^*} \ar[d]_{\alb_X^*}^{\isom} & & H^1_{\et}(\Alb(X), \bQ_\ell) \ar[d]^{\alb_X^*}_{\isom} \\
H^1_{\et}(X, \bQ_\ell) \ar[rr]^{f^*} & & H^1_{\et}(X, \bQ_\ell).
}
\]
This yields that
\[
f^*|_{H^1_{\et}(X, \bQ_\ell)} \sim a_f^*|_{H^1_{\et}(\Alb(X), \bQ_\ell)},
\]
where $\sim$ means that these two linear transformations are similar.
Hence both have the same characteristic polynomial satisfying the required property by \cref{thm:B}.
The case $i=2n-1$ follows readily from Poincar\'e duality.
\end{proof}

Before proving the next corollary, we prepare a lemma on Albert polynomials of symmetric endomorphisms of abelian varieties with respect to the fixed polarization $\phi = \phi_\sL \colon X \to \what X$.

\begin{lemma}
\label{lemma:unique}
Let $\alpha \in \End^0(X)$.
Suppose that $\alpha$ is symmetric under the Rosati involution, i.e., $\alpha^\dagger = \alpha$.
Then its Albert polynomial $P_\alpha^{\Albert}(t) \in \bR[t]$ is unique.
\end{lemma}
\begin{proof}
The lemma basically follows from \cref{thm:B} which asserts the existence of $P_\alpha^{\Albert}(t)$ and the fact that a symmetric/Hermitian matrix has only real eigenvalues.
More precisely, according to \cref{thm:NS}, we could denote the corresponding block diagonal matrix of $\alpha \otimes_\bQ 1_\bR$ by $\bA_\alpha = \bigoplus_i \bA_{\alpha,i}$, where each $\bA_{\alpha,i}$ is either a symmetric real matrix, a Hermitian complex matrix, or a Hermitian quaternionic matrix.
In the first two cases, we know that all eigenvalues of $\bA_{\alpha,i}$ are real numbers.
For the last case, if $\bA_{\alpha,i}\in \sH_r(\bH)$, then the natural embedding $\iota(\bA_{\alpha,i})$ of $\bA_{\alpha,i}$ in $\Mat_{2r}(\bC)$ is a Hermitian complex matrix and hence has only real eigenvalues as well (see, e.g., \cite[Theorem~4]{Lee49}).
Therefore, by \cref{prop:Milne}, all complex roots of the characteristic polynomial $P_\alpha(t)$ are real.
It follows that our Albert polynomial $P_\alpha^{\Albert}(t)$ constructed in \cref{thm:B} lies in $\bR[t]$.
The uniqueness of $P_\alpha^{\Albert}(t)$ thus follows readily.
\end{proof}

\begin{proof}[Proof of \cref{cor:B}]
Thanks to \cref{lemma:unique}, $P_{\alpha^\dagger \circ \alpha}^{\Albert}(t) \in \bR[t]$ is unique because $\alpha^\dagger \circ \alpha = (\alpha^\dagger \circ \alpha)^\dagger$ is symmetric.
Clearly, by the definition of the Albert polynomial (see \cref{def:Albert,rmk:overQ}), we have $P_{\alpha^\dagger \circ \alpha}(t) = P_{\alpha^\dagger \circ \alpha}^{\Albert}(t)^2$.
Note that $P_{\alpha^\dagger \circ \alpha}(t)$ itself lies in $\bQ[t]$.
Then one can easily deduce that $P_{\alpha^\dagger \circ \alpha}^{\Albert}(t) \in \bQ[t]$ by the induction and comparing the coefficients.
Indeed, the constant term $P_{\alpha^\dagger \circ \alpha}(0) = \deg (\alpha^\dagger \circ \alpha) = (\deg \alpha)^2$ is a square (see \cref{thm:char-poly}), which yields that the constant term of $P_{\alpha^\dagger \circ \alpha}^{\Albert}(t)$ is $\pm \deg \alpha \in \bZ$.
The first part of the corollary follows (though one can also see rationality of coefficients from below).

For the second part, note that the characteristic polynomial $P_{\alpha^\dagger \circ \alpha}(t)$ of $\alpha^\dagger \circ \alpha$ is now just the square of its Albert polynomial $P_{\alpha^\dagger \circ \alpha}^{\Albert}(t)$.
As usual, we denote the Euler characteristic of a coherent sheaf $\sF$ of $\sO_X$-modules by $\chi(\sF) \coloneqq \sum_{i=0}^{g} (-1)^i \dim_\bk H^i(X, \sF)$.
Then by the Riemann--Roch theorem (see, e.g., \cite[\S16]{Mumford}), we have
\begin{align*}
\chi(\sL^n \otimes \alpha^* \! \sL^{-1})^2 &= \deg \phi_{\sL^n \otimes \alpha^* \! \sL^{-1}} = \deg (n \, \phi_\sL - \phi_{\alpha^* \! \sL}) \\
&= \deg (n \, \phi_\sL - \what \alpha \circ \phi_\sL \circ \alpha) \\
&= \deg(n \, \phi_\sL - \phi_\sL \circ \alpha^\dagger \circ \alpha) \\
&= \deg \phi_\sL \cdot \deg(n_X - \alpha^\dagger \circ \alpha) \\
&= \chi(\sL)^2 \, P_{\alpha^\dagger \circ \alpha}(n) \\
&= (\chi(\sL) \, P_{\alpha^\dagger \circ \alpha}^{\Albert}(n))^2.
\end{align*}
For the equality $\phi_{\alpha^* \! \sL} = \what \alpha \circ \phi_\sL \circ \alpha$, see \cite[\S15, Theorem~1]{Mumford};
for the second-last equality, see also \cref{thm:char-poly}.
Note that $\sL \coloneqq \sO_X(H_X)$ is an ample line bundle on $X$.
Hence, for sufficiently large $n$, the line bundle $\sL^n \otimes \alpha^* \! \sL^{-1}$ is still ample.
It follows that $\chi(\sL^n \otimes \alpha^* \! \sL^{-1})$ coincides with $\chi(\sL) \, P_{\alpha^\dagger \circ \alpha}^{\Albert}(n)$ as polynomials in $n$ since $P_{\alpha^\dagger \circ \alpha}^{\Albert}(t) (\in \bQ[t])$ is monic.
Applying the Riemann--Roch theorem again, we have
\begin{align*}
P_{\alpha^\dagger \circ \alpha}^{\Albert}(n) &= \frac{\, \chi(\sL^n \otimes \alpha^* \! \sL^{-1}) \, }{\chi(\sL)} = \frac{\, (nH_X - \alpha^* H_X)^g/g! \, }{H_X^g/g!} \\
&= \sum_{k=0}^g (-1)^k \binom{g}{k} \frac{\, \alpha^* H_X^k \cdot H_X^{g-k} \, }{H_X^g} \, n^{g-k}.
\end{align*}
The description about $c_k$ follows by comparing the coefficients.
\end{proof}


\section{Proof of Theorem \ref{thm:A}}
\label{section:proof-A}


The following lemma should be well known.
We include its proof here for the convenience of the reader.
It says that \cref{conj:DDC} has an affirmative answer for complex varieties.

\begin{lemma}
\label{lemma:dg-overC}
Let $f\colon X\ratmap X$ be a dominant rational self-map of a smooth complex projective variety $X$ of dimension $n$.
Then for any $0 \le k \le n$, we have $\chi_{2k}(f)_{\iota} = \lambda_k(f)$.
\end{lemma}
\begin{proof}
By Artin's comparison theorem, in the definition of the $i$-th cohomological dynamical degree $\chi_i$ we can replace the \'etale cohomology $H^i_{\et}(X, \bQ_\ell)$ by the singular/Betti cohomology and hence the de Rham cohomology $H^i_{\dR}(X, \bC)$.
It is well-known that the last one admits the following Hodge decomposition
\[
H^i_{\dR}(X, \bC) = \bigoplus_{p+q=i} H^{p,q}(X, \bC),
\]
where $H^{p,q}(X, \bC) \isom H^q(X, \Omega_X^p)$.
Let $d_k(f)$ denote the usual $k$-th dynamical degree of $f$ in the context of complex dynamics, i.e.,
\begin{align*}
d_k(f) &\coloneqq \limsup_{m\to \infty} \big\|(f^m)^*|_{H^{k,k}(X, \bC)}\big\|^{1/m} \\
&= \lim_{m\to \infty} \big( (f^m)^* \omega_X^k \cdot \omega_X^{n-k} \big)^{1/m},
\end{align*}
where $\omega_X$ is a K\"ahler class on $X$ (see \cite[Corollaire~7]{DS05a} or \cite[Theorem~4.2]{DS17}).
Note that the above $d_k(f)$ does not depend on the choice of $\omega_X$ since the K\"ahler cone is open.
On the other hand, it is known that the $k$-th numerical dynamical degree $\lambda_k(f)$ of $f$ is also equal to
\[
\lim_{m\to \infty} \big( (f^m)^* H_X^k \cdot H_X^{n-k} \big)^{1/m},
\]
where $H_X$ is an ample divisor on $X$ (see \cite{Dang20,Truong20}).
By the openness of the ample cone, our $\lambda_k(f)$ is independent of the choice of $H_X$ too.
This yields that $\lambda_k(f) = d_k(f)$ for any $k$.

It remains to show the equality $d_k(f) = \chi_{2k}(f)$.
First, by the Hodge decomposition, we have
\[
\chi_{2k}(f) = \max_{p+q=2k} r_{p,q}(f) \ge r_{k,k}(f) = d_k(f),
\]
where
\[
r_{p,q}(f) \coloneqq \limsup_{m\to \infty} \big\|(f^m)^*|_{H^{p,q}(X, \bC)}\big\|^{1/m}.
\]
On the other hand, Dinh's inequality \cite[Proposition~5.8]{Dinh05} (see also \cite[Remark~3.3]{DS08}) asserts that $r_{p,q}(f) \le \sqrt{d_p(f) d_q(f)}$.
We thus have
\[
\chi_{2k}(f) \le \max_{p+q=2k} \sqrt{d_p(f) d_q(f)} \le d_k(f).
\]
The last inequality follows from the log-concavity property of dynamical degrees: the function $k \mapsto \log d_k(f)$ is concave in $k$ (see, e.g., \cite[\S4]{DS17}).
Equivalently, we have
\[
d_{k-1}(f)d_{k+1}(f)\le d_k(f)^2, \text{ for any } 1\le k\le n-1.
\]
We thus prove \cref{lemma:dg-overC}.
\end{proof}

\begin{lemma}
\label{thm:P}
Let $X$ be a smooth projective variety of dimension $n$ over $\bk$.
Let $f$ be a polarized endomorphism of $X$, i.e., $f^*H_X \sim_\bQ q H_X$ for some ample divisor $H_X$ on $X$ and an integer $q > 0$.
Then for any $0\le k \le n$, the pullback $f^*$ on $\N^k(X)_\bR$ is diagonalizable with all eigenvalues of modulus $q^k$;
in particular, the $k$-th numerical dynamical degree $\lambda_k(f) = q^k$.
\end{lemma}
\begin{proof}
Let $\Nef^k(X)$ denote the nef cone in $\N^k(X)_\bR$ consisting of nef classes of codimension $k$ on $X$, which is a salient closed convex cone of full dimension preserved by the pullback $f^*$ (see \cite[\S2.3]{FL17a} or \cite[\S2.2.2]{Hu20a}).
Note that the $k$-th self-intersection $H^k$ of the ample divisor $H$ is an interior point in $\Nef^k(X)$ satisfying that $f^*H^k = q^k H^k$ in $\N^k(X)_\bR$.
The lemma thus follows by applying \cite[Proposition~2.9]{MZ18a} to the nef cone $\Nef^k(X) \subset \N^k(X)_\bR$.
\end{proof}

\begin{remark}
\label{rmk:polarized}
In the above lemma, it would be a challenge to show that all eigenvalues of the pullback $f^*$ on $H^i_{\et}(X, \bQ_\ell)$ have absolute value $q^{i/2}$.
When $\bk \subseteq \bC$, this is true by a result due to Serre \cite{Serre60} using Hodge theory.
In positive characteristic, it is a consequence of the standard conjectures (in particular, of Lefschetz type and of Hodge type; see, e.g., \cite[\S2 and \S3]{Kleiman68}, respectively).
See \cite[\S4]{Kleiman68} for more details.
Without assuming the standard conjectures, to the best of our knowledge, it is only known in the following cases:
\begin{enumerate}[(i)]
\item $X$ is a curve or an abelian variety due to Weil's pioneering work \cite{Weil48};
\item $X = X_0 \times_{\bF_q} \ol\bF_q$ for some $X_0$ over $\bF_q$ and $f=\Frob$ is the Frobenius endomorphism of $X$ relative to $\bF_q$ by Deligne \cite{Deligne74}.
\end{enumerate}
Note that it is still open even for arbitrary polarized endomorphisms of ($K3$) surfaces in positive characteristic.
Recently, in joint work with Truong \cite{HT}, we verify this for Kummer surfaces using dynamical correspondences.
\end{remark}

Towards the proof of \cref{thm:A} we still need one more lemma below, which relates the asymptotic behavior of the singular values of the power matrix $\bA^m$ to the spectrum of $\bA$.
See also \cite{Yamamoto67} for a different proof using compound matrices.
Recall that the {\it singular values} of a square matrix $\bA \in \Mat_n(\bC)$ are the square roots of the eigenvalues of $\bA^*\bA$ (or equivalently, $\bA \bA^*$), where $\bA^*$ is the Hermitian transpose of $\bA$.
As a convention, we always denote by $\sigma_i(\bM)$ the $i$-th largest singular value of a general complex matrix $\bM$.

\begin{lemma}[{cf.~\cite[Theorem~1]{Yamamoto67}}]
\label{lemma:non-diag}
Let $\bA \in \Mat_n(\bC)$, whose eigenvalues are $\pi_1, \ldots, \pi_n \in \bC$ so that $|\pi_1| \ge \cdots \ge |\pi_n|$.
For each $m\in \bN$, let $\sigma_1(\bA^m) \ge \cdots \ge \sigma_n(\bA^m)$ denote the singular values of $\bA^m$.
Then for any $1\le i \le n$, we have
\[
\lim_{m\to \infty} \sigma_i(\bA^m)^{1/m} = |\pi_i|.
\]
\end{lemma}
\begin{proof}
It suffices to consider the case when $\bA$ is an upper triangular matrix.
Indeed, let $\bA = \bU \bT \bU^*$ be the Schur decomposition of $\bA$ (see, e.g., \cite[Theorem 2.3.1]{HJ13}), where $\bU$ is unitary and $\bT$ is upper triangular with diagonal entries $\pi_1, \ldots, \pi_n$.
Then $\bA^m = \bU \bT^m \bU^*$ so that $\bA^m (\bA^m)^* = \bU \bT^m (\bT^m)^* \bU^*$ for any $m \in \bN$.
Hence, $\sigma_i(\bA^m) = \sigma_i(\bT^m)$ for any $i$.
We now assume that $\bA = \bT$ is upper triangular.
Without loss of generality, we may assume further that $\pi_n \ne 0$, i.e., $\bT$ is non-singular, since otherwise we would easily have $\sigma_n(\bT^m) = 0$.

We first prove that $\displaystyle \lim_{m\to \infty} \sigma_i(\bT^m)^{1/m} = |\pi_i|$ for $i = 1$ and $n$.
Note that $\sigma_1(\bT^m)$ is equal to the $2$-norm (also known as the spectral norm) $\|\bT^m\|_2$ of $\bT^m$ (see, e.g., \cite[Example~5.6.6]{HJ13}).
Thus the well-known spectral radius formula asserts that
\[
\lim_{m\to \infty} \sigma_1(\bT^m)^{1/m} = \lim_{m\to \infty} \big\|\bT^m\big\|_2^{1/m} = \rho(\bT) = |\pi_1|.
\]
On the other hand, we can repeat the above argument to the inverse $\bT^{-m}$ of $\bT^m$.
More precisely, we note that the largest eigenvalue of $(\bT^{-m})^* \bT^{-m}$, as the inverse of $\bT^m (\bT^{m})^*$, is just $\sigma_n(\bT^m)^{-2}$.
In other words, the largest singular value of $\bT^{-m}$, which coincides with the $2$-norm $\|\bT^{-m}\|_2$ of $\bT^{-m}$, is equal to $\sigma_n(\bT^m)^{-1}$.
It follows that
\[
\lim_{m\to \infty} \sigma_n(\bT^m)^{1/m} = \lim_{m\to \infty} \big\|\bT^{-m}\big\|_2^{-1/m} = \rho(\bT^{-1})^{-1} = |\pi_n|.
\]

In particular, the lemma has been proved when $n = 2$.
We shall prove the general case by induction on the matrix size $n$.
Let $n \ge 3$.
Suppose that the lemma holds for any matrix in $\Mat_{n-1}(\bC)$.
By the preceding discussion, it remains to consider the intermediate case $i=2, \ldots, n-1$.
First of all, let us fix the following notation: for a general matrix $\bM$, we use $\bM_{k}$ to denote the principal submatrix of $\bM$ obtained by removing both $k$-th row and $k$-th column from $\bM$.
Therefore, we can rewrite $\bT$ as follows:
\[
\bT = 
\begin{pmatrix}
\pi_1 & \bu_1^\sT \\
0 & \bT_1
\end{pmatrix}
\ \text{ or } \ 
\begin{pmatrix}
\bT_n & \bv_n \\
0 & \pi_n
\end{pmatrix},
\]
where $\bT_1$ (resp. $\bT_n$) is an upper triangular matrix with $\pi_2, \ldots, \pi_n$ (resp. $\pi_1, \ldots, \pi_{n-1}$) on the diagonal.
It is easy to verify that $\bT_1^m = (\bT^m)_1$ and $\bT_n^m = (\bT^m)_n$ for any $m$, since $\bT$ is upper triangular.
As usual, let $\sigma_1(\bT_1^m) \ge \cdots \ge \sigma_{n-1}(\bT_1^m)$ and $\sigma_1(\bT_n^m) \ge \cdots \ge \sigma_{n-1}(\bT_n^m)$ denote the singular values of $\bT_1^m$ and $\bT_n^m$, respectively.
Now, for each $m \in \bN$, applying Cauchy's interlacing theorem to the principal submatrix $\bT_1^m$ of $\bT^m$ (see, e.g., \cite[Theorem~4.3.17]{HJ13}), we get
\[
\sigma_1(\bT^m) \ge \sigma_1(\bT_1^m) \ge \sigma_2(\bT^m) \ge \sigma_2(\bT_1^m) \ge \cdots \ge \sigma_{n-1}(\bT^m) \ge \sigma_{n-1}(\bT_1^m) \ge \sigma_n(\bT^m).
\]
The induction hypothesis asserts that $\displaystyle \lim_{m\to \infty} \sigma_i(\bT_1^m)^{1/m} = |\pi_{i+1}|$ for any $1 \le i \le n-1$.
Taking the limsup of the $m$-th roots of the above sequences, we thus obtain that
\begin{equation}
\label{eq:limsup}
|\pi_2| \ge \limsup_{m\to \infty} \sigma_2(\bT^m)^{1/m} \ge |\pi_3| \ge \cdots \ge \limsup_{m\to \infty} \sigma_{n-1}(\bT^m)^{1/m} \ge |\pi_n|.
\end{equation}
Similarly, the same argument works for $\bT_n^m$, which yields the following inequalities
\[
\sigma_1(\bT^m) \ge \sigma_1(\bT_n^m) \ge \sigma_2(\bT^m) \ge \sigma_2(\bT_n^m) \ge \cdots \ge \sigma_{n-1}(\bT^m) \ge \sigma_{n-1}(\bT_n^m) \ge \sigma_n(\bT^m).
\]
It thus follows from the induction hypothesis that
\begin{equation}
\label{eq:liminf}
|\pi_1| \ge \liminf_{m\to \infty} \sigma_2(\bT^m)^{1/m} \ge |\pi_2| \ge \cdots \ge \liminf_{m\to \infty} \sigma_{n-1}(\bT^m)^{1/m} \ge |\pi_{n-1}|.
\end{equation}
Combining inequalities \eqref{eq:limsup} and \eqref{eq:liminf}, we have shown that $\displaystyle \lim_{m\to \infty} \sigma_i(\bT^m)^{1/m} = |\pi_i|$ for all intermediate $i = 2, \ldots, n-1$.
This concludes the proof of \cref{lemma:non-diag}.
\end{proof}

We are now ready to prove our main theorem on the comparison of the cohomological dynamical degrees with the numerical ones on abelian varieties, extending the main result of \cite{Hu19}.

\begin{proof}[Proof of \cref{thm:A}]
It suffices to consider the case when $f$ is a surjective endomorphism of $X$.
Indeed, any morphism between abelian varieties is a composite of a homomorphism with a translation (see, e.g., \cite[Corollary~2.2]{Milne86}).
Hence we can write $f$ as $t_a \circ \alpha$ for a surjective endomorphism $\alpha\in \End(X)$ and a translation $t_a$ for some $a\in X(\bk)$.
Note however that $t_a$ acts as identity on $H^1_{\et}(X, \bQ_\ell)$ and hence on $H^i_{\et}(X, \bQ_\ell)$ for all $i$.
It follows from the functoriality of the pullback map on $\ell$-adic \'etale cohomology that $\chi_i(f) = \chi_i(\alpha)$.
Similarly, we also get that $\lambda_k(f) = \lambda_k(\alpha)$ for all $k$.
So from now on, our $f = \alpha \in \End(X)$ is an isogeny.

Let $P_\alpha(t) \in \bZ[t]$ be the characteristic polynomial of $\alpha$.
By \cref{thm:B}, we can denote its $2g$ complex roots by $\pi_1, \ldots, \pi_g, \ol\pi_1, \ldots, \ol\pi_g$.
Without loss of generality, we may assume that
\[
|\pi_1| \ge \cdots \ge |\pi_g| > 0.
\]
It thus follows from \cref{thm:etale-coh} and the spectral radius formula that the $2k$-th cohomological dynamical degree of $\alpha$ is
\begin{equation}
\label{eq:chi}
\chi_{2k}(\alpha)_{\iota} = \prod_{i=1}^{k} |\pi_i|^2.
\end{equation}

We shall use \cref{cor:B} to compute $\lambda_k(\alpha)$ on the other side.
For each $m \in \bN$, let us first consider the characteristic polynomial $P_{(\alpha^m)^\dagger \circ \alpha^m}(t)$ of the symmetric element $(\alpha^m)^\dagger \circ \alpha^m$.
In virtue of \cref{prop:MilneII}, we have
\[
P_{\alpha}(t) = \prod_{j=1}^s \chi_{\alpha_j}^{\reduced}(t)^{m_j} \quad \text{and} \quad P_{(\alpha^m)^\dagger \circ \alpha^m}(t) = \prod_{j=1}^s \chi_{(\alpha_j^m)^\dagger \circ \alpha_j^m}^{\reduced}(t)^{m_j}.
\]
Note that by the definition, for each $j$, the reduced characteristic polynomial $\chi_{\alpha_j}^{\reduced}(t)$ is nothing but the characteristic polynomial of the corresponding matrix $\bA_{\alpha_j}$ of $\alpha_j \otimes_\bQ 1_\bC \in \End(X_j)_\bC$;
it is also the characteristic polynomial of $\alpha_j$ acting on the reduced representation $V_j^{\reduced}$.
Similarly, $\chi_{(\alpha_j^m)^\dagger \circ \alpha_j^m}^{\reduced}(t)$ is the characteristic polynomial of the Hermitian matrix $\bA_{(\alpha_j^m)^\dagger \circ \alpha_j^m} = (\bA_{\alpha_j}^m)^*\bA_{\alpha_j}^m$.
In particular, apart from multiplicities, the $\pi_i$ coincide with the eigenvalues of $\bA_\alpha = \bigoplus_j \bA_{\alpha_j}$,
and the roots of $P_{(\alpha^m)^\dagger \circ \alpha^m}(t)$ the eigenvalues of $(\bA_\alpha^m)^*\bA_\alpha^m$ (i.e., the squares of the singular values of $\bA_\alpha^m$).
Thus, without loss of generality, for each $m \in \bN$, we can denote the $2g$ real roots of $P_{(\alpha^m)^\dagger \circ \alpha^m}(t)$ by
$\sigma_1(\alpha^m)^2, \ldots, \sigma_g(\alpha^m)^2, \sigma_1(\alpha^m)^2, \ldots, \sigma_g(\alpha^m)^2$, where
\[
\sigma_1(\alpha^m) \ge \cdots \ge \sigma_g(\alpha^m) > 0
\]
coincide with the singular values of $\bA_\alpha^m$ (apart from multiplicities).

However, we note that for each $j$, the multiplicity $m_j$ depends only on $X_j$ (or rather, $X_j$ and $\End^0(X_j)$, but not the index $m$).
It thus follows from \cref{lemma:non-diag} that for each $1 \le i \le g$,
\begin{equation}
\label{eq:limit}
\lim_{m\to \infty} \sigma_i(\alpha^m)^{1/m} = |\pi_i|.
\end{equation}
According to \cref{lemma:unique}, the Albert polynomial of $(\alpha^m)^\dagger \circ \alpha^m$ can be written as
\[
P_{(\alpha^m)^\dagger \circ \alpha^m}^{\Albert}(t) = \prod_{i=1}^g (t - \sigma_i(\alpha^m)^2 ).
\]
Hence applying \cref{cor:B} to $\alpha^m$ yields that
\begin{equation}
\label{eq:sym-poly}
\binom{g}{k} \frac{\, (\alpha^m)^* H_X^k \cdot H_X^{g-k}\, }{H_X^g} = e_k ( \sigma_1(\alpha^m)^2, \ldots, \sigma_g(\alpha^m)^2 ),
\end{equation}
where $e_k$ is the $k$-th elementary symmetric polynomial.
Note that by \cite[Lemma~2.6]{Hu20a}, the $k$-th numerical dynamical degree $\lambda_k(\alpha)$ of $\alpha$ can be reinterpreted by the formula
\begin{equation}
\label{eq:lambda}
\lambda_k(\alpha) = \lim_{m\to \infty} \left( (\alpha^m)^* H_X^k \cdot H_X^{g-k} \right)^{1/m};
\end{equation}
see also \cite[Theorems 1 and 2]{Dang20} and \cite[Theorem~1.1(1)]{Truong20} for more general cases of dominant rational maps and dominant correspondences, respectively.

Putting together all \cref{eq:sym-poly,eq:lambda,eq:limit,eq:chi}, we have that
\begin{align*}
\lambda_k(\alpha) &= \lim_{m\to \infty} \left( e_k( \sigma_1(\alpha^m)^2, \ldots, \sigma_g(\alpha^m)^2 ) \right)^{1/m} \\[3pt]
&= \max_{1\le i_1 < \ldots < i_k \le g} \, \lim_{m\to \infty} \sigma_{i_1}(\alpha^m)^{2/m} \cdots  \lim_{m\to \infty} \sigma_{i_k}(\alpha^m)^{2/m} \\[3pt]
&= \max_{1\le i_1 < \ldots < i_k \le g} \, |\pi_{i_1}|^2 \cdot |\pi_{i_2}|^2 \cdots |\pi_{i_k} |^2 \\[3pt]
&= |\pi_1|^2 \cdot |\pi_2|^2 \cdots |\pi_k|^2 \\
&= \chi_{2k}(\alpha)_{\iota}.
\end{align*}
We thus complete the proof of \cref{thm:A}.
\end{proof}

\begin{remark}
\label{rmk:A2}
With notation as above, one can easily verify that if $i = 2k$ is even, then 
\[
\chi_i(\alpha)_{\iota}^2 = \prod_{j=1}^{k} |\pi_j|^4 = \lambda_{k}(\alpha)^2 = \max_{p+q=2k} \lambda_p(\alpha) \lambda_q(\alpha);
\]
and if $i = 2k - 1$ is odd, then
\[
\chi_i(\alpha)_{\iota}^2 = |\pi_k|^2 \cdot \prod_{j=1}^{k-1} |\pi_j|^4= \lambda_{k-1}(\alpha) \lambda_{k}(\alpha) = \max_{p+q=2k-1} \lambda_p(\alpha) \lambda_q(\alpha).
\]
This yields an analog of Dinh's inequality and hence answers \cite[Question~4]{Truong} in the case of abelian varieties.
\end{remark}

In the end, as a by-product, we deduce the following norm comparison for self-morphisms of abelian varieties, which motivates the so-called norm comparison conjecture considered in \cite{HT}.
Note that by the triangle inequality, the same also holds for any correspondence $f\coloneqq \sum_{i=1}^n a_i \Gamma_{f_i}$, where $a_i\in \bQ_{>0}$ and $\Gamma_{f_i}$ is the graph of a self-morphism $f_i$ of $X$.

\begin{corollary}
\label{cor:norm}
Let $X$ be an abelian variety of dimension $g$ over $\bk$.
Then there exists a positive constant $C>0$ such that for any $0\le k\le g$ and for any self-morphism $f$ of $X$, we have
\[
\big\|f^*|_{H^{2k}_{\emph\et}(X, \bQ_\ell)}\big\|_{\iota} \le C \, \big\|f^*|_{\N^k(X)_\bR}\big\|.
\]
\end{corollary}
\begin{proof}
It suffices to consider the case when $f = \alpha$ is an endomorphism of $X$.
Also, by the equivalence of norms on the finite-dimensional vector spaces, we are free to choose any norms.
It is well-known that the spectral norm $\|\alpha^*|_{\N^k(X)_\bR}\|_2$ of $\alpha^*|_{\N^k(X)_\bR}$ is equal to $\sigma_1(\alpha^*|_{\N^k(X)_\bR})$, the largest singular value of $\alpha^*|_{\N^k(X)_\bR}$.
The latter turns out to be the square root of the spectral radius $\rho((\alpha^{\dagger}\circ \alpha)^*|_{\N^k(X)_\bR})$,
since $(\alpha^{\dagger})^*|_{\N^k(X)_\bR}$ is represented by the transpose of $\alpha^*|_{\N^k(X)_\bR}$ which is defined over $\bZ$.
It thus follows from \cref{thm:A} that
\begin{equation}
\label{eq:2-norm-on-Nk}
\big\|\alpha^*|_{\N^k(X)_\bR}\big\|_2 = \rho( (\alpha^{\dagger}\circ \alpha)^*|_{\N^k(X)_\bR} )^{1/2} = \rho( (\alpha^{\dagger}\circ \alpha)^*|_{H^{2k}_{\et}(X, \bQ_\ell)} )^{1/2}.
\end{equation}

We shall prove that the right-hand side of \cref{eq:2-norm-on-Nk} also gives a norm of $\alpha^*|_{H^{2k}_{\et}(X, \bQ_\ell)}$.
Notice that by \cref{thm:etale-coh}, one has
\[
\alpha^*|_{H^{2k}_{\et}(X, \bQ_\ell)} = \bigwedge\nolimits^{\! 2k} \alpha^*|_{H^{1}_{\et}(X, \bQ_\ell)} \isom \bigwedge\nolimits^{\! 2k} V_\ell(\alpha)^\vee |_{V_\ell(X)^\vee}.
\]
On the other hand, by \cref{prop:MilneII}, we have the following decomposition of $V_\ell(X) \otimes_{\bQ_\ell} \ol{\bQ}_\ell$ as a representation of $\End^0(X)$ (for brevity, here we suppress the multiplicity $m_j$):
\[
V_\ell(X) \otimes_{\bQ_\ell} \ol{\bQ}_\ell \isom \bigoplus_{j} V_j^{\reduced} \otimes_{\ol{\bQ}} \ol{\bQ}_\ell,
\]
which yields that
\[
V_\ell(\alpha) |_{V_\ell(X) \otimes_{\bQ_\ell} \ol{\bQ}_\ell} \isom \alpha |_{\bigoplus_{j} V_j^{\reduced} \otimes_{\ol{\bQ}} \ol{\bQ}_\ell} \isom \bigoplus_{j} \alpha_j |_{V_j^{\reduced} \otimes_{\ol{\bQ}} \ol{\bQ}_\ell}.
\]
For each $j$, we know that $\alpha_j|_{V_j^{\reduced}}$ is represented by a matrix $\bA_{\alpha_j} \in \Mat_{e_jd_jn_j}(\ol{\bQ})$; further, viewed as a complex matrix, $\bA_{\alpha_j}$ is nothing but the corresponding matrix of $\alpha_j \otimes_\bQ 1_\bC \in \End(X_j)_\bC$.
Then we see from \cref{thm:NS} that $\alpha_j^{\dagger}|_{V_j^{\reduced}}$ is represented by the Hermitian transpose $\bA_{\alpha_j}^*$ of $\bA_{\alpha_j}$.
It follows that the linear maps $\alpha^*|_{H^{2k}_{\et}(X, \bQ_\ell)}$ and $(\alpha^\dagger \circ \alpha)^*|_{H^{2k}_{\et}(X, \bQ_\ell)}$ could be represented by the complex matrices
\[
\bigwedge\nolimits^{\! 2k} \bigoplus_{j} \bA_{\alpha_j}^\sT \quad \mbox{and} \quad
\bigwedge\nolimits^{\! 2k} \bigoplus_{j} \ol\bA_{\alpha_j} \bA_{\alpha_j}^\sT = \bigg(\bigwedge\nolimits^{\! 2k} \bigoplus_{j} \bA_{\alpha_j}^\sT \bigg)^* \cdot \bigg(\bigwedge\nolimits^{\! 2k} \bigoplus_{j} \bA_{\alpha_j}^\sT \bigg),
\]
respectively.
This thus yields that the spectral radius of $(\alpha^{\dagger}\circ \alpha)^*|_{H^{2k}_{\et}(X, \bQ_\ell)}$ is just the square of the largest singular value of $\alpha^*|_{H^{2k}_{\et}(X, \bQ_\ell)}$.
In other words, we show that
\[
\rho( (\alpha^{\dagger}\circ \alpha)^*|_{H^{2k}_{\et}(X, \bQ_\ell)} )^{1/2} = \sigma_1(\alpha^*|_{H^{2k}_{\et}(X, \bQ_\ell)}) = \big\|\alpha^*|_{H^{2k}_{\et}(X, \bQ_\ell)}\big\|_2.
\]
This concludes the proof of \cref{cor:norm}.
\end{proof}

\begin{remark}
It is worth mentioning that in \cite{Zarhin20}, Zarhin proves that for an endomorphism $\alpha$ of an abelian variety $X$ of dimension $g$ over $\bk$, the matrix representing $V_\ell(\alpha)|_{V_\ell(X)}$ (or equivalently, $\alpha^*|_{H^{1}_{\et}(X, \bQ_\ell)}$), lies in $\Mat_{2g}(\bQ)$, under an appropriate basis of $V_\ell(X)$.
Later, Poonen and Rybakov \cite{PR21} refine his result to integer matrices.
Hence the above $2$-norm argument applies to the pullback action on \'etale cohomology as well, which immediately shows that the right-hand side of \cref{eq:2-norm-on-Nk} gives a norm of $\alpha^*|_{H^{2k}_{\et}(X, \bQ_\ell)}$.
\end{remark}


\phantomsection
\addcontentsline{toc}{section}{Acknowledgments}
\noindent
\textbf{Acknowledgments. }
I am greatly obliged to Dragos Ghioca and Zinovy Reichstein for their support during my stay at the University of British Columbia.
I would like to thank H\'el\`ene Esnault, Nathan Grieve, Bruno Kahn, Nicholas Katz, Zinovy Reichstein, Tuyen Truong, Yuri Zarhin, Yishu Zeng, and De-Qi Zhang for stimulating discussions and valuable suggestions.
I would also like to thank Pierre Deligne for kindly explaining to me that our \cref{thm:B} is actually a consequence of the standard conjectures.
My special thanks goes to Tuyen Truong for helpful comments and for pointing out several inaccuracies in an earlier version.
Finally, I am grateful to the referee for his/her many helpful and invaluable suggestions.



\bibliographystyle{amsalpha}
\bibliography{../mybib}

\end{document}